\documentclass[12pt,reqno]{amsart}
\usepackage{amsopn,amssymb,mathrsfs,url,mathrsfs,a4wide}

\newcommand{\ball}[1]{\ensuremath{B_{#1}}}
\newcommand{\cl}[1]{\ensuremath{\overline{{#1}}}}
\newcommand{\diam}[1]{\ensuremath{\sdiam{\left({#1}\right)}}}
\newcommand{\dual}[1]{\ensuremath{{#1}^*}}
\newcommand{\ep}{\varepsilon}
\newcommand{\itp}[2]{\ensuremath{{#1}\hat{\otimes}_{\varepsilon}{#2}}}
\newcommand{\lspan}[1]{\ensuremath{\aspan({#1})}}
\newcommand{\map}[3]{\ensuremath{{#1}:{#2}\to{#3}}}
\newcommand{\mpi}{\ensuremath{\mathrm{\pi}}}
\newcommand{\n}[1]{\ensuremath{\left\|{#1}\right\|}}
\newcommand{\N}{\mathbb{N}}
\newcommand{\ndot}{\ensuremath{\left\|\cdot\right\|}}
\newcommand{\oneton}[2]{\ensuremath{{#1}_1,\ldots,{#1}_{#2}}}
\newcommand{\R}{\mathbb{R}}
\newcommand{\set}[2]{\ensuremath{\left\{{#1}\;:\;\,{#2}\right\}}}
\newcommand{\tn}[1]{\ensuremath{\left|\kern-.9pt\left|\kern-.9pt\left|{#1}\right|\kern-.9pt\right|\kern-.9pt\right|}}
\newcommand{\tndot}{\ensuremath{\left|\kern-.9pt\left|\kern-.9pt\left|\cdot\right|\kern-.9pt\right|\kern-.9pt\right|}}
\newcommand{\tp}[2]{\ensuremath{{#1}\otimes{#2}}}
\newcommand{\ttri}{|\kern-.9pt|\kern-.9pt|}
\newcommand{\ttrin}{\ttri\cdot\ttri}
\newcommand{\weakstar}{\ensuremath{w^*}}
\newcommand{\wone}{\ensuremath{\omega_1}}

\DeclareMathOperator{\aspan}{span}
\DeclareMathOperator{\card}{card}
\DeclareMathOperator{\conv}{conv}
\DeclareMathOperator{\ext}{ext}
\DeclareMathOperator{\sdiam}{diam}
\DeclareMathOperator{\supp}{supp}

\newtheorem{thm}{Theorem}[section]
\newtheorem{cor}[thm]{Corollary}
\newtheorem{lem}[thm]{Lemma}
\newtheorem{prop}[thm]{Proposition}

\theoremstyle{definition}
\newtheorem{defn}[thm]{Definition}
\newtheorem{example}[thm]{Example}
\newtheorem{rem}[thm]{Remark}

\begin{document}
\title[Topology, isomorphic smoothness and polyhedrality in Banach spaces]{Topology, isomorphic smoothness and polyhedrality in Banach spaces}
\begin{abstract}
In recent decades, topology has come to play an increasing role in some geometric aspects of Banach space theory. The class of so-called \emph{$w^*$-locally relatively compact} sets was introduced recently by Fonf, Pallares, Troyanski and the author, and were found to be a useful topological tool in the theory of isomorphic smoothness and polyhedrality in Banach spaces. We develop the topological theory of these sets and present some Banach space applications.
\end{abstract}

\author{Richard J.~Smith}
\address{School of Mathematics and Statistics, University College Dublin, Belfield, Dublin 4, Ireland}
\email{richard.smith@maths.ucd.ie}
\urladdr{http://mathsci.ucd.ie/~rsmith}

\thanks{The author was supported financially by Science Foundation Ireland under Grant Number `SFI 11/RFP.1/MTH/3112'.}

\date{\today}
\maketitle

\section{Introduction}\label{sect_introduction}

Given $k \in \N \cup \{\infty\}$, we say that a norm $\ndot$ on a Banach space $X$ is {\em $\mathscr{C}^k$-smooth} if it is $k$-times continuously Frech\'et differentiable on $X\setminus\{0\}$. A norm $\ndot$ is said to {\em depend locally on finitely many coordinates from $H \subseteq X^*$} ({\em LFC-$H$} for short) if, given $x \in X\setminus\{0\}$, there exists an open set $U \ni x$, functionals $\oneton{f}{n} \in H$, and a map $\map{\Phi}{\R^n}{\R}$, such that
\[
\n{y} \;=\; \Phi(f_1(y),\dots, f_n(y)),
\]
whenever $y \in U$. If $\ndot$ is LFC-$X^*$ then we simply call it LFC. This notion was first explicitly introduced and investigated in \cite{pwz:81}. Many $\mathscr{C}^k$-smooth norms happen to be LFC, see e.g.\ \cite{hajek:95,hh:07} and the survey \cite{hz:06}, because it is much easier to construct smooth functions locally on finite-dimensional spaces and then glue them together, rather than build smooth functions directly on infinite-dimensional space (a notable exception being the canonical norm on Hilbert space).

We say that a norm $\ndot$ on a Banach space $X$ is {\em polyhedral} if, given any finite-dimensional subspace $E$, the unit ball $B_E$ of $E$ with respect to $\ndot$ is a polytope, that is, there exist $\oneton{f}{n} \in X^*$ such that
\[
\n{y} \;=\; \max\{f_1(y),\dots, f_n(y)\},
\]
whenever $y \in E$. The theory of polyhedral norms was instigated by Klee \cite{klee:59} and major contributions were made later by Fonf in a series of papers, including \cite{fonf:80,fonf:81,fonf:00}.

Polyhedral norms and LFC norms share several properties. The next theorem gathers results from \cite[Theorem 5]{fonf:80}, \cite{pwz:81}, \cite[Corollary 2]{fz:97} and \cite[Theorem 3.9]{fonf:00}.

\begin{thm}[\cite{fonf:80,pwz:81,fz:97,fonf:00}]\label{thm_nec_conditions}
Let the norm of $X$ be either polyhedral or LFC. Then
\begin{enumerate}
\item $X$ is an Asplund space, which is equivalent to saying that every separable subspace of $X$ has separable dual, and
\item $X$ is $c_0$-saturated, i.e., every infinite-dimensional subspace of $X$ contains an isomorphic copy of $c_0$.
\end{enumerate}
\end{thm}

The notion of a boundary of a Banach space (or the norm of the space) is central to the theory of both polyhedral and LFC norms. A set $B \subseteq B_{X^*}$ is called a {\em boundary} of the Banach space $(X,\ndot)$ if, whenever $x \in X$, there exists $f \in B$ satisfying $f(x)=\n{x}$. Sometimes this is referred to as a {\em James boundary} in the literature. The properties of boundaries are not preserved by isomorphisms in general:~boundaries of $(X,\ndot)$ and $(X,\tndot)$, where $\tndot$ is an equivalent norm, may be very different in character.

Standard examples of boundaries include the dual unit sphere $S_{X^*}$ of $X^*$, and the set of extreme points $\ext(B_{X^*})$ of $B_{X^*}$, by the Hahn-Banach and Krein-Milman Theorems. If $(X,\ndot)$ has a boundary that is `small' (e.g.\ countable) or otherwise well-behaved, then $X$ enjoys good geometric properties as a consequence.

In particular, in the separable case, we obtain a series of equivalent conditions that tie together the above notions very tightly, up to isomorphism. Hereafter, if a Banach space (with a given norm) is said to admit a norm, then it should be understood that this norm is equivalent to the original one.

\begin{thm}[\cite{fonf:90,hajek:95}]\label{sepequiv} Let $X$ be a separable Banach space. The following
statements are equivalent.
\begin{enumerate}
\item $X$ admits a norm $\ndot$ such that $(X,\ndot)$ has a countable boundary;
\item $X$ admits a norm $\ndot$ such that $(X,\ndot)$ has a norm-$\sigma$-compact boundary;
\item $X$ admits a polyhedral norm;
\item $X$ admits a polyhedral LFC norm;
\item $X$ admits a LFC norm;
\item $X$ admits a $\mathscr{C}^\infty$-smooth LFC norm.
\end{enumerate}
\end{thm}

Moreover, by Theorem \ref{thm_nec_conditions}, if $X$ is separable and satisfies any of the conditions above then it its dual is separable. In general, the relationship between the existence of polyhedral norms, LFC norms and smooth (LFC) norms is less clear. The following notion, involving a pair of topologies, was introduced recently in \cite{fpst:14}, and was used to provide a sufficient condition for the existence of equivalent polyhedral and smooth LFC norms.

\begin{defn}[{\cite[Definition 11]{fpst:14}}]\label{lrc}
Let $X$ be a set and let $\tau$ and $\rho$ be two Hausdorff topologies on $X$, with $\rho$ finer than $\tau$. We say that $E \subseteq X$ is {\em $\tau$-locally relatively $\rho$-compact} ($(\tau,\rho)$-LRC for short), if given $x \in E$, there exists a $\tau$-open set $U \subseteq X$, such that $x \in U$ and $\cl{E \cap U}^{\rho}$ is $\rho$-compact. We say that $E$ is {\em $\sigma$-$(\tau,\rho)$-LRC} if it can be expressed as a countable union of $(\tau,\rho)$-LRC sets.
\end{defn}

The interplay between pairs of topologies can be found in other topological notions employed in Banach space theory, such as sets of {\em small local diameter} \cite[p.~162]{jnr:92}, and Raja's property $P(\cdot,\cdot)$ (see Section \ref{sect_LUR}).

In the main, we shall be concerned with ($w^*,\ndot$)-LRC subsets of dual Banach spaces. For brevity, we shall simply call these $w^*$-LRC sets. Recall that a {\em Markushevich basis} (M-basis) $(e_\gamma,e_\gamma^*)_{\gamma \in \Gamma}$ of a Banach space $X$ is a biorthogonal system, such that $X=\cl{\aspan}^{\ndot}(e_\gamma)_{\gamma \in \Gamma}$ and $X^* = \cl{\aspan}^{w^*}(e^*_\gamma)_{\gamma \in \Gamma}$. Given $x \in X$ and $f \in X^*$, define their {\em supports}
\[
\supp(x) \;=\; \set{\gamma \in \Gamma}{e^*_\gamma(x) \neq 0}\quad\text{and}\quad \supp(f) \;=\; \set{\gamma \in \Gamma}{f(e_\gamma) \neq 0}.
\]

\begin{example}[{\cite[Example 6 and Proposition 20]{fpst:14}}]\label{weakstarlrcex}~
\begin{enumerate}
\item Any relatively norm-compact or $w^*$-relatively discrete subset of a dual Banach space is $w^*$-LRC.
\item Any norm-$K_\sigma$ set is $\sigma$-$w^*$-LRC.
\item Let $(e_\gamma,e_\gamma^*)_{\gamma \in \Gamma}$ be a M-basis of a Banach space $X$. Given a fixed $n \in \N$, the set
\[
\set{f \in X^*}{\card(\supp f) = n},
\]
is $w^*$-LRC, and thus $\set{f \in X^*}{\supp f \text{ is finite}}$ is $\sigma$-$w^*$-LRC.
\item If $K$ is a compact space that is \emph{$\sigma$-discrete} (i.e.~the union of countably many sets, each of which is relatively discrete), then the set of Dirac measures $\set{\delta_t}{t \in K} \subseteq C(K)^*$ is $\sigma$-$w^*$-LRC.
\end{enumerate}
\end{example}

It was shown in \cite[p.\ 254]{fpst:14} that the unit sphere $S_{X^*}$ of an infinite-dimensional dual Banach space $X^*$ can never be $\sigma$-$w^*$-LRC (this applies equally to $S_Y$, where $Y \subseteq S_{X^*}$ is any infinite-dimensional subspace, and follows from Proposition \ref{prop_lrc_facts} (1) below). However, it is possible for certain boundaries, such as $\ext(B_{X^*})$ in some cases, to be $\sigma$-$w^*$-LRC. In applications, our boundaries will be simultaneously $w^*$-LRC and $w^*$-$K_\sigma$ (a countable union of $w^*$-compact sets). Our principal application follows after the next definition.

\begin{defn}\label{epapprox}
Let $(X,\ndot)$ be a Banach space and let $\ep>0$. We will say that a new norm $\tndot$ on $X$ {\em $\ep$-approximates}
$\ndot$ if 
\[
\n{x} \;\leqslant\; \tn{x} \;\leqslant\; (1+\ep)\n{x}, 
\]
for all $x \in X$. Moreover, given a property $\mathbf{P}$ of norms, we say that $\ndot$ can be {\em approximated} by norms having $\mathbf{P}$ if, given $\ep>0$ there is a norm $\tndot$ having $\mathbf{P}$ that $\ep$-approximates $\ndot$.
\end{defn}

\begin{thm}[{\cite[Theorem 2.1]{bible:14}} and {\cite[Theorem 7]{fpst:14}}]\label{tb}
Let $(X,\ndot)$ be a Banach space having a boundary $B$ that is both $\sigma$-$w^*$-LRC and $w^*$-$K_\sigma$. Then $\ndot$ can be approximated by both $C^\infty$-smooth and polyhedral norms.
\end{thm}

Theorem \ref{tb} generalises Theorem \ref{sepequiv} (2) $\Rightarrow$ (4), (6), and reduces to these implications in the separable case, because if $X$ is separable, then $(\ball{\dual{X}},w^*)$ is hereditarily Lindel\"of, which implies that a $\sigma$-$w^*$-LRC set is a countable union of relatively norm-compact sets. This theorem also generalises results from \cite{fpst:08,hh:07}.

Examples of spaces satisfying the hypotheses of Theorem \ref{tb} include certain Orlicz and Nakano spaces, preduals of certain Lorentz spaces and $C(K)$-spaces, where $K$ is compact and $\sigma$-discrete  -- see \cite[Example 3.9]{bible:14} and \cite[Examples 16 and 18, and Corollary 19]{fpst:14}.

More recently, Theorem \ref{tb} has also been used to show that an arbitrary equivalent norm on $c_0(\Gamma)$ can be approximated by both $\mathscr{C}^\infty$-smooth and polyhedral norms \cite{bs:16}. This is the first non-separable approximation result of its kind. Even more recently, this result was generalised to more classes of Banach spaces, namely certain subclasses of the Orlicz, Nakano, Lorentz predual and $C(K)$-spaces mentioned in the previous paragraph \cite{st:18}.

In this context, the most optimistic outcome would be to formulate a non-separable analogue of Theorem \ref{sepequiv} (perhaps restricted to the class of Banach spaces admitting a norm having a locally uniformly rotund dual norm -- see Section \ref{sect_LUR}). The class of sets that are both $\sigma$-$w^*$-LRC and $w^*$-$K_\sigma$ naturally generalize the norm-$\sigma$-compact sets in a manner suitable for the construction of $\mathscr{C}^\infty$-smooth and polyhedral norms in a non-separable setting. In the author's opinion, this class represents an important step in the direction of such an analogue, and believes that further study of the properties of the class is desirable. 

\section{Topological permanence properties of ($\tau,\rho$)-LRC sets}\label{sect-top}

In this section, we use topological methods to make some observations about ($\tau,\rho$)-LRC sets, specifically concerning permanence properties, that is, the preservation (or otherwise) of these sets under topological and linear operations. These observations will then yield geometric consequences. We begin by repeating some known topological facts concerning these sets that are used in this paper. 

\begin{prop}[{\cite[Proposition 12]{fpst:14}}]\label{prop_lrc_facts}
Let $X$, $\tau$ and $\rho$ be as in Definition \ref{lrc}.
\begin{enumerate}
\item If $(X,\tau)$ is a Baire space and $\cl{U}^\rho$ is not $\rho$-compact whenever $U$ is $\tau$-open and non-empty, then $X$ is not $\sigma$-$(\tau,\rho)$-LRC.
\item If $\rho$ is metrizable (with metric also denoted by $\rho$), then any $(\tau,\rho)$-LRC set $E\subseteq X$ has small local diameter, in the sense that given $x \in E$ and $\ep>0$, there exists a $\tau$-open subset $U \subseteq X$, such that $x \in E \cap U$ and $\rho$-$\diam{E \cap U} < \ep$.
\item If $E\subseteq X$ is $(\tau,\rho)$-LRC then there exists a $\tau$-open set $V$, such that $E \subseteq \cl{E}^\tau \cap V \subseteq \cl{E}^\rho$, and $\cl{E}^\tau \cap V$ is also $(\tau,\rho)$-LRC.
\end{enumerate}
\end{prop}

In order to consider permanence properties of ($\tau,\rho$)-LRC sets, we must introduce some compactness. Without some form of compactness in play, $\sigma$-$w^*$-LRC sets are not preserved under the most simple linear operations. Indeed, in \cite[Remark 1.10 (3)]{bible:14}, an example is given of a relatively $w^*$-discrete subset $E$ of a dual Banach space $X^*$, such that $E+E$ includes the unit ball of an infinite-dimensional subspace $Y \subseteq X^*$. Consequently, $E+E$ is not $\sigma$-$w^*$-LRC (see the remarks after Example \ref{weakstarlrcex}). We remark also that the compactness assumption is vital for Theorem \ref{tb} to operate correctly as well:~in a development of \cite[Remark 1.10 (3)]{bible:14}, Bible showed that the space $Z:=\ell_1 \oplus \ell_1(S_{\ell_\infty})$, which is neither Asplund nor $c_0$-saturated (because it contains $\ell_1$), and thus not isomorphically polyhedral, by Theorem \ref{thm_nec_conditions}, nevertheless admits a norm with respect to which $Z$ has a relatively $w^*$-discrete boundary \cite[Remark 3.4.4]{bible:16}. Therefore, simply being $\sigma$-$w^*$-LRC is not a sufficient condition for the existence of the types of norms under discussion.

So a little compactness seems necessary for any sort of reasonable behaviour, and it turns out we can prove some satisfactory permanence properties once we are armed with it. Theorem \ref{perfectimage} below shows that the property of being $\sigma$-$(\tau,\rho)$-LRC is preserved by certain types of perfect maps. Let $\tau$ and $\rho$ and $\tau'$ and $\rho'$ be two pairs of Hausdorff topologies on two sets $X$ and $Y$, respectively, with $\rho$ finer than $\tau$ and $\rho'$ finer than $\tau'$.

\begin{thm}\label{perfectimage}
Let $X$ be $\sigma$-$(\tau,\rho)$-LRC, and let $\map{\pi}{X}{Y}$ be a surjective map that is $\tau$-$\tau'$ perfect and $\rho$-$\rho'$ continuous. Then $Y$ is $\sigma$-$(\tau',\rho')$-LRC.
\end{thm}

Using this theorem, we are able to prove a series of further results and corollaries.

\begin{thm}\label{linspan}
If $E$ is a $\sigma$-$\weakstar$-LRC and $\weakstar$-$K_\sigma$ subset of a dual Banach space $X^*$, then so is $\lspan{E}$.
\end{thm}

As a particular application of Theorem \ref{linspan}, if the set $E \subseteq X^*$ possesses a single $w^*$-accumulation point $0$, then $E \cup \{0\}$ is the union of two $w^*$-discrete sets $E$ and $\{0\}$, and countably many $w^*$-compact sets $\set{f \in E}{\n{f} \leqslant n} \cup \{0\}$, $n \in \N$. Consequently, $\lspan{E}$ is $\sigma$-$\weakstar$-LRC and $\weakstar$-$K_\sigma$. Hence the linear structure of the M-bases underlying Example \ref{weakstarlrcex} (3) is not needed:~given such a basis $(e_\gamma,e^*_\gamma)$, we see that $0$ is the only $w^*$-accumulation point of $\set{e^*_\gamma}{\gamma \in \Gamma}$.

The next corollary is used in \cite[Section 7]{st:18}, as part of the proof of the result mentioned in Section \ref{sect_introduction}, concerning the approximation by $\mathscr{C}^\infty$-smooth and polyhedral norms of arbitrary equivalent norms on certain $C(K)$-spaces.

\begin{cor} Let $K$ be a $\sigma$-discrete compact space. Then the linear subspace of measures in $C(K)^*$ supported on finitely many points of $K$ is $\sigma$-$w^*$-LRC.
\end{cor}

\begin{proof} Apply Theorem \ref{linspan} to the linear span of the subset in Example \ref{weakstarlrcex} (4), which is $w^*$-compact.
\end{proof}

The next corollary follows immediately, using Theorem \ref{tb}.

\begin{cor}\label{bdy-lin-span}
Let $E\subseteq X^*$ be $\sigma$-$\weakstar$-LRC and $\weakstar$-$K_\sigma$, and suppose that $(X,\ndot)$ admits a boundary contained in $\lspan{E}$. Then $X$ admits norms as in Theorem \ref{tb}.
\end{cor}

This leads to two further corollaries.

\begin{cor}[cf.\ {\cite[Corollary 3.5]{bible:14} and \cite[Corollary 14]{fpst:14}}]\label{M-basis} If $(e_\gamma,e^*_\gamma)_{\gamma \in \Gamma}$ is an M-basis of $X$, and $(X,\ndot)$ has a boundary contained in $\aspan(e^*_\gamma)_{\gamma \in \Gamma}$, then $X$ admits norms as in Theorem \ref{tb}.
\end{cor}

\begin{proof}
Apply Corollary \ref{bdy-lin-span}, together with the example after Theorem \ref{linspan}.
\end{proof}

The final corollary has implications for LFC norms. 

\begin{cor}\label{LFC}
Let $X$ be a Banach space and let $\ndot$ be a LFC-$E$ norm, where $E\subseteq X^*$ is $\sigma$-$\weakstar$-LRC and $\weakstar$-$K_\sigma$. Then $X$ admits norms as in Theorem \ref{tb}.
\end{cor}

\begin{proof}
In the proof of \cite[Theorem 1]{hajek:95}), it is shown that if a norm $\ndot$ is LFC-$E$, where $E \subseteq X^*$ is some set, then every norm-attaining element of $S_{X^*}$ is an element of $\aspan(E)$. Consequently, $X$ has a boundary contained in $\aspan(E)$. If $E$ is $\sigma$-$\weakstar$-LRC and $\weakstar$-$K_\sigma$, we can thus apply Corollary \ref{bdy-lin-span}.
\end{proof}

In order to prove Theorems \ref{perfectimage} and \ref{linspan}, we require some machinery that is based on Haydon's analysis of locally uniformly rotund norms on $C(K)$, where $K$ is a so-called Namioka--Phelps compact space \cite{haydon:08}. Lemma \ref{minimal} and its proof can be found in \cite{ks:15}, but it is essentially due to Haydon. Let $X$ be a Hausdorff space, such that $X=\bigcup_{n=1}^\infty H_n$, where each $H_n$ is open in its closure. Let
\[
\Sigma \;=\; \set{\sigma=(n_1,n_2,\dots,n_k)}{n_1 < n_2 < \dots < n_k,\, k \in \N}.
\]
Totally order $\Sigma$ be declaring that $\sigma \prec \sigma'$ if and only if $\sigma$ properly extends $\sigma'$, or if there exists $k \in \N$ such that the $i$th entries $n_i$ and $n'_i$ of $\sigma$ and $\sigma'$, respectively, are defined for $i \leqslant k$, agree whenever $i < k$, and $n_k < n'_k$. This is the {\em Kleene--Brouwer ordering} on $\Sigma$ and not the lexicographic order.

Given $\sigma=(n_1,n_2,\dots,n_k) \in \Sigma$, let
\begin{align*}
H_\sigma &\;=\; (\cl{H_{n_1}}\setminus H_{n_1}) \cap \dots \cap (\cl{H_{n_{k-1}}}\setminus H_{n_{k-1}}) \cap H_{n_k}\\
\text{and}\quad \hat{H}_\sigma &\;=\; (\cl{H_{n_1}}\setminus H_{n_1}) \cap \dots \cap (\cl{H_{n_{k-1}}}\setminus H_{n_{k-1}}) \cap \cl{H_{n_k}}.
\end{align*}
Evidently, $H_\sigma \subseteq \hat{H}_\sigma$ and $\hat{H}_\sigma$ is closed.

\begin{lem}[{\cite[Lemma 3.6]{ks:15}, \cite[Lemma 3.3]{haydon:08}}]\label{minimal}
Let $M \subseteq X$ be non-empty and compact. Then there exists minimal $\sigma \in \Sigma$ such that $M \cap \hat{H_\sigma}$ is non-empty. Moreover, for this element $\sigma$, we have $M \cap H_\sigma = M \cap \hat{H}_\sigma$.
\end{lem}

Armed with this result, we can present the proof of Theorem \ref{perfectimage}.

\begin{proof}[Proof of Theorem \ref{perfectimage}]
Let $X=\bigcup_{n=1}^\infty H_n$, where each $H_n$ is $(\tau,\rho)$-LRC. By Proposition \ref{prop_lrc_facts} (3), we can assume that each $H_n$ is $\tau$-open in $\cl{H_n}^\tau$. Let $\Sigma$ be as above, and let
\[
M_\sigma \;=\; \set{t \in Y}{\pi^{-1}(t) \cap H_\sigma = \pi^{-1}(t) \cap \hat{H}_\sigma \neq \varnothing}.
\]
Since $\pi^{-1}(t)$ is non-empty and $\tau$-compact, by Lemma \ref{minimal}, we know that $Y = \bigcup_{\sigma\in\Sigma} M_\sigma$.
Our aim now is to prove that each $M_\sigma$ is $(\tau',\rho')$-LRC. Fix $t \in M_\sigma$. Given any $s \in \pi^{-1}(t) \cap \hat{H}_\sigma$, there exists $\tau$-open $U_s \ni s$ such that $\cl{U_s \cap H_\sigma}^\rho$ is $\rho$-compact. By $\tau$-compactness, there exist $\oneton{s}{k} \in \pi^{-1}(t) \cap \hat{H}_\sigma$ such that if $U = \bigcup_{i=1}^k U_{s_i}$, then $\pi^{-1}(t) \cap \hat{H}_\sigma \subseteq U$. Moreover,
\[
\cl{U \cap H_\sigma}^\rho \;=\; \cl{\bigg(\bigcup_{i=1}^k U_{s_i} \cap H_\sigma\bigg)}^\rho
\;=\; \bigcup_{i=1}^k \cl{U_{s_i} \cap H_\sigma}^\rho,
\]
which is $\rho$-compact.

Now $\hat{H}_\sigma\setminus U$ is $\tau$-closed, so $W=Y\setminus \pi(\hat{H}_\sigma\setminus U)$ is $\tau'$-open in $M$ because $\pi$ is $\tau$-$\tau'$ perfect. We claim that $t \in W$ and $\cl{W \cap M_\sigma}^{\rho'}$ is $\rho'$-compact. Indeed, $t \in W$ because $\pi^{-1}(t) \cap \hat{H}_\sigma\setminus U = \varnothing$, which implies $t \notin \pi(\hat{H}_\sigma\setminus U)$. Given that $\pi$ is $(\rho,\rho')$-continuous, the second assertion follows from the fact that $W \cap M_\sigma \subseteq \pi(U \cap H_\sigma)$. If $r \in W \cap M_\sigma
=M_\sigma\setminus\pi(\hat{H}_\sigma\setminus U)$, then $\pi^{-1}(r) \cap H_\sigma = \pi^{-1}(r) \cap \hat{H}_\sigma \neq \varnothing$, and $r \notin \pi(\hat{H}_\sigma\setminus U)$, so $\pi^{-1}(r)\cap \hat{H}_\sigma\setminus U = \varnothing$, i.e.\ $\pi^{-1}(r)\cap \hat{H}_\sigma \subseteq U$. Pick an element $v \in \pi^{-1}(r)\cap \hat{H}_\sigma = \pi^{-1}(r) \cap H_\sigma$. Then $v \in U \cap H_\sigma$, giving $r=\pi(v) \subseteq \pi(U \cap H_\sigma)$.
\end{proof}

We require two more results before we can give the proof of Theorem \ref{linspan}.

\begin{cor}\label{ksigma-image}
Let $\map{\mpi}{X}{Y}$ be a $\tau$-$\tau'$ and $\rho$-$\rho'$ continuous map, and let $E \subseteq X$ be $\sigma$-$(\tau,\rho)$-LRC and $\sigma$-$\tau$-compact. Then $\pi(E)$ is $\sigma$-$(\tau',\rho')$-LRC and $\sigma$-$\tau'$-compact.
\end{cor}

\begin{proof} If $E = \bigcup_{n=1}^\infty K_n$, where each $K_n$ is $\tau$-compact, then $\pi$ restricted to $K_n$ is $\tau$-$\tau'$ perfect onto its image, so $\pi(K_n)$ is $\sigma$-$(\tau',\rho')$-LRC, and is evidently $\tau'$-compact.
\end{proof}

The next result we state without proof.

\begin{prop}\label{prodLRC}
Let $X_i$, $1 \leqslant i \leqslant n$, be sets and $\tau_i$ and $\rho_i$ Hausdorff topologies on $X_i$, with $\rho_i$ finer than $\tau_i$. If $E_i \subseteq X_i$ is $(\tau_i,\rho_i)$-LRC, then $\prod_{i=1}^n E_i$ is $(\tau,\rho)$-LRC, where $\tau$ and $\rho$ are the product topologies of the $\tau_i$ and $\rho_i$, respectively.
\end{prop}

\begin{proof}[Proof of Theorem \ref{linspan}]
Since $E$ is $\weakstar$-$K_\sigma$, we can express $E$ as a union of an increasing sequence of $\weakstar$-compact sets $K_n$, $n \geqslant 1$, each of which being $\sigma$-$\weakstar$-LRC. Then $\lspan{E} = \bigcup_{n=1}^\infty \lspan{K_n}$, so it suffices to show that $\lspan{K_n}$ is $\sigma$-$\weakstar$-LRC and $\weakstar$-$K_\sigma$ for each $n$. Fix $n$ for the rest of the proof, and given $p \in \N$, define
\[
M_p \;=\; K_n^p \times [-p,p]^p \quad\text{and}\quad L_p \;=\; \set{\sum_{i=1}^p a_i f_i}{\oneton{f}{p} \in K_n \text{ and }|a_1|,\dots,|a_p| \leqslant p}.
\]
Since $K_n$ is $\sigma$-$\weakstar$-LRC and $[-p,p]$ is LRC with respect to the usual topology, it follows from Proposition \ref{prodLRC} that $M_p$ is $\sigma$-$(\tau,\rho)$-LRC, where $\tau$ are $\rho$ are the corresponding product topologies. Evidently, the surjective map $\map{\pi}{M_p}{L_p}$ given by
\[
\pi(f_1,\dots,f_p,a_1,\dots,a_p) \;=\; \sum_{i=1}^p a_i f_i,
\]
is both $\tau$-$\weakstar$ and $\rho$-$\ndot$ continuous. Since $M_p$ is $\tau$-compact, $\pi$ is $\tau$-$\weakstar$ perfect as well. According to Theorem \ref{perfectimage}, $L_p$ is $\sigma$-$\weakstar$-LRC. This holds for all $p \in \N$. Since $L_p$ is also $\weakstar$-compact, it follows that $\lspan{K_n} = \bigcup_{p=1}^\infty L_p$ is $\sigma$-$\weakstar$-LRC and $\weakstar$-$K_\sigma$, as required.
\end{proof}

We conclude the first part of this section with two more corollaries. The first corollary is motivated by the fact that isomorphic polyhedrality and $\mathscr{C}^\infty$-smoothness are properties preserved by isomorphic embeddings.

\begin{cor}
Let $X$ and $Y$ be Banach spaces and $\map{T}{X}{Y}$ an isomorphic embedding. If $(Y,\ndot)$ has a boundary that is contained in a $\sigma$-$\weakstar$-LRC and $\weakstar$-$K_\sigma$ set, then so does $(X,\tndot)$, where the equivalent norm $\tndot$ is defined by
$\tn{x} = \n{T(x)}$.
\end{cor}

\begin{proof}
Apply Corollary \ref{ksigma-image} to $\map{T^*}{Y^*}{X^*}$. If $B$ is a boundary of $Y$ with respect to $\ndot$, then it is easy to see that $T^*(B)$ is a boundary of $X$ with respect to $\tndot$.
\end{proof}

Finally, injective tensor products have been studied in the context of smooth and polyhedral norms. Let $X$ and $Y$ be Banach spaces, where the norm on $Y$ is either polyhedral or $\mathscr{C}^k$-smooth. In \cite{bible:14,fpst:08,haydon:96}, sufficient conditions on $X$ are given for the injective tensor product $\itp{X}{Y}$ to admit a norm of the same type. One such condition is that $X$ has a boundary that is both $\sigma$-$w^*$-LRC and $w^*$-$K_\sigma$. This observation motivates the final result of the section.

\begin{cor}
If $X$ and $Y$ are Banach space have norms that admit $\sigma$-$w^*$-LRC and $w^*$-$K_\sigma$ boundaries, then so does the natural tensor norm of $\itp{X}{Y}$. 
\end{cor}

\begin{proof}
Let $B$ and $C$ be $\sigma$-$w^*$-LRC and $w^*$-$K_\sigma$ boundaries of the norms on $X$ and $Y$, respectively. Then $B\times C$ is a $\sigma$-$w^*$-LRC and $w^*$-$K_\sigma$ subset of $X^*\oplus Y^*$. The set
\[
\set{\tp{f}{g}}{(f,g) \in B\times C} \;\subseteq \; (\itp{X}{Y})^*,
\]
is a boundary of $\itp{X}{Y}$ with respect to the natural injective tensor norm. Moreover, it is a simultaneously norm and $w^*$-continuous image of $B\times C$, so it is $\sigma$-$w^*$-LRC and $w^*$-$K_\sigma$ by Corollary \ref{ksigma-image}.
\end{proof}

%
%

\section{The role of locally uniformly rotund dual norms}\label{sect_LUR}

In this section, we use more topological tools to give a necessary condition for the existence of a $\sigma$-$w^*$-LRC and $w^*$-$K_\sigma$ boundary. Recall that a norm $\ndot$ is {\em locally uniformly rotund (LUR)} if, given a unit vector $x \in X$ and a sequence $(x_n) \subseteq X$ such that $\n{x_n} \to 1$ and $\n{x+x_n}\to 2$, we have $\n{x-x_n}\to 0$.

The theory of LUR norms has become closely associated with topology (see \cite{motv:09}). The following definition appears in \cite{or:04,raja:02} (and in earlier papers of Raja). Let $K \subseteq X^*$ be a $w^*$-compact subset of a dual Banach space. We say that $K$ has \emph{$P(\ndot,w^*)$ with closed sets} if there is a sequence $(A_n)_{n=1}^\infty$ of $w^*$-closed subsets of $K$, such that given $f \in K$ and $\ep>0$, there exist $n\in\N$ and a $w^*$-open set $U$ with the property that $x \in A_n \cap U$ and $\ndot$-$\diam{A_n \cap U}<\ep$. Equivalently, $K$ as above has $P(\ndot,w^*)$ with closed sets if and only if $(K,w^*)$ is \emph{descriptive} and \emph{fragmented} by the norm \cite{or:04}. We shall require the following theorem.

\begin{thm}[{\cite[Theorems 1.3, 2.2 and 2.5]{or:04}}]\label{or_thm}
Let $X$ be a Banach space, let $K \subseteq X^*$ have $P(\ndot,w^*)$ with closed sets, and let $X^* = \cl{\aspan}^{\ndot}(K)$. Then $X$ admits a norm having LUR dual norm.
\end{thm}

It is well known that if $X$ has a separable dual then $X$ admits a norm having LUR dual norm (see e.g.~\cite[Theorem II.2.6]{dgz:93}). This result is an easy corollary of Theorem \ref{or_thm}. Indeed, if $X^*$ is separable, with $(f_n)$ norm-dense in $S_{X^*}$, then $X^*=\cl{\aspan}^{\ndot}(K)$, where $K:=\set{n^{-1}f_n}{n \in \N} \cup \{0\}$, and it is straightforward to see that $K$ has $P(\ndot,w^*)$ with closed sets. We use Theorem \ref{or_thm} to prove the main result of this section.

\begin{thm}\label{main}
Let $X$ admit a boundary that is both $\sigma$-$w^*$-LRC $w^*$-$K_\sigma$. Then $X$ admits a norm having a LUR dual norm. 
\end{thm}

As indicated in Section \ref{sect_introduction}, this suggests to the author that the class of Banach spaces admitting norms having LUR dual norms is the correct setting for the study of Banach spaces supporting $\sigma$-$w^*$-LRC and $w^*$-$K_\sigma$ boundaries. We obtain the following corollary, which is an optimal strengthening of \cite[Proposition 20]{fpst:14}.

\begin{cor}\label{sigma-discrete}
Let $C(K)$ admit a norm $\tndot$ having a $\sigma$-$w^*$-LRC and $w^*$-$K_\sigma$ boundary. Then $K$ is $\sigma$-discrete.
\end{cor}

\begin{proof}
According to \cite[Corollary 4.4]{raja:02}, $C(K)$ admits a norm having LUR dual norm if and only if $K$ is $\sigma$-discrete.
\end{proof}

The next example demonstrates that $\sigma$-$w^*$-LRC and $w^*$-$K_\sigma$ boundaries are not necessary for isomorphic polyhedrality in general.

\begin{example}\label{ex_counterexample}
If $K=\wone+1$, where $\wone$ is the least uncountable ordinal, then $K$ is not $\sigma$-discrete. On the other hand, $C(K)$ admits a polyhedral norm \cite[Theorem 1]{fonf:80}. Therefore there exists an isomorphically polyhedral Banach space that does admit any norm having a $\sigma$-$w^*$-LRC and $w^*$-$K_\sigma$ boundary.
\end{example}

We require two more results in order to prove Theorem \ref{main}. The first is a consequence of a result of Fonf and Lindenstrauss, concerning boundaries of $\weakstar$-compact convex sets.

\begin{thm}[{cf.\ \cite[Theorem 2.3]{fl:03}}]\label{fl_thm}
Let $B \subseteq \ball{\dual{X}}$ be a boundary of $X$, and let $B = \bigcup_{n=1}^\infty C_n$. Then
\[
\ball{\dual{X}} \;=\; \cl{\conv}^{\ndot}\left(\bigcup_{n=1}^\infty \cl{\conv}^{\weakstar}(C_n) \right).
\]
\end{thm}

The next result comes from \cite{raja:02}, where it is stated in greater generality.

\begin{prop}[{\cite[Lemma 2.1 and Proposition 2.10]{raja:02}}]\label{prop_raja_P_conv} Let $X$ be a Banach space and let $K \subseteq X^*$ be a $w^*$-compact set having $P(\ndot,w^*)$ with closed sets. Then $\cl{\conv}^{w^*}(K)$ also has $P(\ndot,w^*)$ with closed sets.
\end{prop}

\begin{proof}[Proof of Theorem \ref{main}]
Suppose that $E = \bigcup_{n=1}^\infty E_n \subseteq \ball{\dual{X}}$ is a $w^*$-$K_\sigma$ boundary of $X$, where each $E_n$ is $w^*$-LRC. We can assume that $\cl{E_n}^{w^*} \subseteq E$ for all $n$ (if necessary, consider the sets $E_n \cap L_m$, where $E= \bigcup_{m=1}^\infty L_m$ and each $M_m$ is $w^*$-compact). Furthermore, by Proposition \ref{lrc} (3), we can assume that there are $w^*$-open sets $V_n$ such that $E_n = \cl{E_n}^{w^*} \cap V_n$.

Set $K_n = \cl{E_n}^{w^*}$. Our first objective is to show that $K_n$ has $P(\ndot,w^*)$ with closed sets. Fix $n$ and set $A_m = K_n \cap K_m$, $m\in\N$. Given $f \in K_n \subseteq E$ and $\ep>0$, locate $m\in\N$ such that $f \in E_m = K_m \cap V_m$. According to Proposition \ref{prop_lrc_facts} (2), there exists a $w^*$-open set $U$, which we can assume to be a subset of $V_m$, such that $f \in E_m \cap U$ and $\diam{E_m \cap U} < \ep$. Consequently, $f \in A_m \cap U$ and $A_m \cap U \subseteq K_m \cap V_m \cap U = E_m \cap U$. It follows that $K_n$ has $P(\ndot,w^*)$ with closed sets, as required.

By Theorem \ref{fl_thm}, we have
\[
B_{X^*} \;=\; \cl{\conv}^{\ndot}\left(\bigcup_{n=1}^\infty \cl{\conv}^{w^*}(K_n) \right).
\]
From this, it is easy to see that if we define the $w^*$-compact set
\[
K \;=\; \{0\} \cup \bigcup_{n=1}^\infty n^{-1}\cl{\conv}^{w^*}(K_n),
\]
then $\dual{X} = \cl{\aspan}^{\ndot}(K)$. Moreover, each set $\cl{\conv}^{w^*}(K)$ has $P(\ndot,w^*)$ with closed sets by Proposition \ref{prop_raja_P_conv} and, being a countable union of such sets, it is clear that $K$ has this property also. We conclude the proof by invoking Theorem \ref{or_thm}.
\end{proof}

\begin{rem}
As it happens, if the set $E$ satisfies the properties given in the proof of
Theorem \ref{main}, and is in addition symmetric (which we can easily assume by considering $E \cup (-E)$ if necessary), then $\ball{\dual{X}} = \cl{\conv}^{\ndot}(E)$. Thus, in the proof above, it is sufficient to consider $K = \{0\} \cup \bigcup_{n=1}^\infty n^{-1} K_n$.

Indeed, the proof of \cite[Theorem 7]{fpst:14} shows that for a given $\ep>0$, there exists a function $\map{\psi}{E}{(1,1+\ep)}$ and a subset $B \subseteq E$, such that if we set $D=\set{\psi(f)f}{f \in B}$ and $\tn{x}=\sup\set{g(x)}{g \in D}$, then $\tndot$ is a polyhedral norm, $\n{x} \leqslant \tn{x}$ for all $x \in X$, and $D$ is a boundary of $\tndot$.

According to \cite[Theorem 3.9]{fonf:00} and \cite[Theorem 2]{vesely:00}, the polyhedrality of $\tndot$ ensures that $\ball{\dual{(X,\ttrin)}} =
\cl{\conv}^{\ndot}(D)$. Therefore, we have the inclusions
\[
\ball{\dual{X}} \;\subseteq\; \ball{\dual{(X,\ttrin)}} \;=\; \cl{\conv}^{\ndot}(D)
\;\subseteq\; (1+\ep)\cl{\conv}^{\ndot}(E),
\]
from which we obtain
\[
(1+\ep)^{-1}\ball{\dual{X}} \;\subseteq\; \cl{\conv}^{\ndot}(E),
\]
for all $\ep>0$. The result follows.
\end{rem}

\end{document}